\documentclass[11pt]{amsart}
\usepackage{amssymb,amsxtra,amsmath,amsfonts, accents}
\usepackage{verbatim,enumerate}
\usepackage{graphicx,color,graphics} 
\usepackage[all]{xy}
\usepackage{braket}
\usepackage[colorlinks = true,
  linkcolor = blue,
 urlcolor  = green,
citecolor = red,
anchorcolor = blue]{hyperref}
\usepackage{comment}

\usepackage{mathtools}
\usepackage{tikz}
\usetikzlibrary{chains}

\tikzset{node distance=2em, ch/.style={circle,draw,on chain,inner sep=2pt},chj/.style={ch,join},every path/.style={shorten >=4pt,shorten <=4pt},line width=1pt,baseline=-1ex}

\usetikzlibrary{decorations.markings}

\tikzset{
  midarrow/.style={
    postaction={decorate},
    decoration={markings, mark=at position 0.5 with {\arrow{>}}}
  }
}

\newcommand{\alabel}[1]{}

\newtheorem{theorem}{Theorem}
\newtheorem{lemma}[theorem]{Lemma}
\newtheorem{proposition}[theorem]{Proposition}

\theoremstyle{definition}
\newtheorem{definition}[theorem]{Definition}
\newtheorem{example}[theorem]{Example}

\theoremstyle{remark}

\numberwithin{equation}{section}
\numberwithin{theorem}{section}

\renewcommand{\comment}[1]{}
\def\CC{\mathbb{C}}

\def\V{\mathcal{V}}

\def\C{\mathbb{C}}

\def\M{\mathcal{M}}

\def\T{\mathcal{T}}

\def\ZZ{\mathbb{Z}}

\DeclareMathOperator\ad{ad}

\DeclareMathOperator\Tr{Tr}

\DeclareMathOperator\Span{span}

\def\lieg{{\mathfrak{g}}}
\def\lieh{{\mathfrak{h}}}

\makeindex

\title[Lie conformal algebras and one loop corrections of amplitudes]{Non-linear Lie conformal algebras and one-loop corrections of self-dual Yang-Mills amplitudes }
\author{Charles Igel, Jeremy Mancinas and Juan Villarreal}

\begin{document}

\maketitle
\begin{abstract} This work is motivated by recent developments in celestial holography. In \cite{CP}, the authors interpreted QCD collinear singularities in terms of operator product expansions in a two-dimensional CFT.  We reformulate the algebraic structures arising in their work using the formalism of non-linear Lie conformal algebras developed in \cite{SK}.
\end{abstract}
\section{Introduction}

Recently, new examples of infinite-dimensional algebras have been appearing in the physics literature on the celestial holography program \cite{ S, PPR, GHPS, CP}, see also \cite{BHPZ, ABC+}. In particular, an approach based on twistor theory \cite{W,  CP, CP2, C, B} explains how conformal blocks of certain vertex algebras give rise to expressions equivalent to scattering amplitudes \cite{PT, BDDK}. These striking relations have motivated us to study such algebras from a mathematical perspective using the theory of non-linear Lie conformal algebras \cite{SK}.

In this work, we consider the vertex algebra $\V^{\beta}(\lieg)$ studied in \cite{CP}, defined in \eqref{eq v}, where $\mathfrak{g}$ is a finite-dimensional simple complex Lie algebra and $\beta \in \C$. This vertex algebra has infinitely many generators, labeled as follows:
\begin{equation}\label{eq: gen}
\{J_a[n,m], I_a[n,m], E[n,m], F[n,m]\}_{a \in \mathfrak{g},\, n,m \geq 1}.
\end{equation}
where $E[0,0]=0$. In \cite{CP2}, the authors, motivated by the theory of collinear limits in one-loop QCD scattering amplitudes, deform the vertex algebra $\V^{\beta}(\lieg)$ by introducing, in particular, a set of parameters $D, C \in \CC$ and relations \eqref{eq v2} on a finite subset of generating fields given by
\begin{equation}\label{eq: gen2}
\{J_a[n,m], I_a[n,m]\}_{a \in \mathfrak{g},\, n+m \leq 1}\, . 
\end{equation}

In the present work, we reformulate those constructions in the language of non-linear Lie conformal algebras. Additionally, we have removed some constants which are physically motivated in their work.  In this work, we recover a main result of \cite{CP}, which is our main result and we state it as follows.

\begin{theorem} \label{theorem: 1.1}The following Jacobi identity is satisfied
  \[[J_a[1,0]_\lambda [J_b[0,1]_\mu J_c[0,0]]] =[J_b[0,1]_\mu [J_a[1,0]_\lambda J_c[0,0]]]+[[J_a[1,0]_\lambda J_b[0,1]]_{\lambda+\mu}J_c[0,0]]\]
     if and only if $\beta=D=C=0$ or $\lieg\in\{A_1, A_2, D_4,  E_6, E_7, E_8, F_4, G_2\}$ and
     \begin{align*}
  D=\frac{-\beta^2(2+\dim \lieg )}{20h^{\vee}}\, ,  \qquad C=\frac{3\beta^2(2+\dim \lieg)}{20(h^{\vee})^2}\, .
\end{align*}
\end{theorem}

We remark that general recursive expressions for the brackets of all generators in \eqref{eq: gen} were obtained in \cite{FP}. In future work, we plan to reinterpret certain constructions in \cite{FP} and relate them to the existence theorem in \cite[Theorem 3.9]{SK}. Moreover, some of the algebras arising in celestial holography are more general than ordinary vertex algebras, as they exhibit logarithmic singularities \cite{BLRSY, BS, CB, Kr}. We hope to further connect these physical constructions with the theory of logarithmic vertex algebras developed in \cite{G, G2, GL, BV, BV2, BV3, V}.

\subsection*{Acknowledgments}
We thank the University of Colorado Boulder Department of Mathematics for hosting us during their 2025 summer REU. We are especially grateful for the funding, space, meals, and sense of community provided at the beginning of this project, all of which helped make this research possible. 

\section{Preliminaries}

In this work $\lieg$ denotes a finite-dimensional complex simple Lie algebra endowed with the bi-invariant pairing 
\begin{equation}\label{fq2.1} (a,b)=\frac{1}{2h^\vee}\Tr(\ad_{a}\ad_b)\, ,\end{equation}
 where $a,b\in\lieg$ and $h^\vee$ denotes the dual Coxeter number. For an introduction to Lie conformal algebras see \cite{K,A, AM}.

\begin{definition}
    A {Lie conformal algebra} is a $\mathbb{C}[\partial]$-module $R$ endowed with an $\C$-bilinear map $[\cdot_\lambda \cdot] : R \otimes R \rightarrow R [\lambda] $ satisfying
    \begin{enumerate}
        \item[] {Sesquilinearity}\,\,\, $[\partial a_\lambda b] = -\lambda [a_\lambda b], \hspace{.25cm}[a_\lambda \partial b] = (\lambda + \partial) [a_\lambda b]$,
        \item[] {Skew-symmetry} $[a_\lambda b] = -[b_{-\lambda-\partial} a]$,
        \item[] {Jacobi identity}\,\, $[a_\lambda [b_\mu c]] = [b_\mu [a_\lambda c]] + [[a_\lambda b]_{\lambda+\mu} c]$.
   \end{enumerate}
   where  $a,b,c\in R$.
\end{definition}

We have the following well-known example.
\begin{example}
    Let $\lieg$ be a Lie algebra. 
    The {affine conformal algebra} $R_{\lieg}=\C [T] \otimes \lieg$ is defined by the bracket
        \begin{equation}\label{eq1.1}
        [a_\lambda b] = [a,b] 
        \end{equation}
        for all $a,b\in \lieg$.       
\end{example}
We remark that the previous example has a central extension, but in this work we are mainly interested in \eqref{eq1.1}.


\begin{definition} A \emph{vertex algebra} is a quintuple $(V, \ket{0}, \partial, ., [\cdot_\lambda \cdot])$ satisfying
 \begin{itemize}
        \item[i)] $(V, \partial, [\cdot_\lambda \cdot])$ is a Lie conformal algebra.
        \item[ii)] $(V, \partial, .)$ is a quasicommutative, quasiassociative unital differential algebra, see \cite{K}
        \item[iii)] The product $.$ and bracket $ [\cdot_\lambda \cdot]$ are related by the noncommutative Wick formula
        \begin{equation}\label{eq2.0}
        [a_\lambda bc] = [a_\lambda b]c + b[a_\lambda c] + \int_0^\lambda [[a_\lambda b]_\mu c]d\mu
        \end{equation}
    \end{itemize}

\end{definition}

From a Lie conformal algebra $R$ there is a canonical way to construct a
vertex algebra which contains $R$, see \cite{K, BK}

\begin{theorem}\label{thm BK} Let $R$ be a Lie conformal algebra. Let $R_{\text{Lie}}$ be $R$ considered as a Lie algebra over $\mathbb{C}$ with respect to the Lie bracket
\[
[a, b] = \int^{0}_{-\partial}  [a_\lambda b]d\lambda \qquad a, b \in R,
\]
and let $V = U(R_{\text{Lie}})$ be its universal enveloping algebra. Then there exists a unique structure of a vertex algebra on $V$ such that the restriction of the $\lambda$-product to $R_{\text{Lie}} \times R_{\text{Lie}}$ coincides with the $\lambda$-bracket on $R$, and the restriction of the normally ordered product to $R_{\text{Lie}} \times V$ coincides with the associative product of $U(R_{\text{Lie}})$.
\end{theorem}

Let $U$ be a vector space. The tensor algebra
\[
\T(U) := \mathbb{C} \oplus U \oplus U^{\otimes 2} \oplus \cdots\, , 
\]
is naturally a vector space. 
Moreover if $U$ is a $\mathbb{C}[T]$-module $\T(U)$ is a $\mathbb{C}[T]$-module by letting $T(1)=0$ and $T(A\otimes B)=T(A)\otimes B+A\otimes T(B)$.
\begin{definition}
    A non-linear {conformal algebra} is a $\mathbb{C}[T]$-module $R$ endowed with an $\C$-bilinear map $[\cdot_\lambda \cdot] : R \otimes R \rightarrow \T(R) [\lambda] $ satisfying
    \begin{enumerate}
        \item[] {Sesquilinearity}\,\,\, $[\partial a_\lambda b] = -\lambda [a_\lambda b], \hspace{.25cm}[a_\lambda \partial b] = (\lambda + \partial) [a_\lambda b]$.
   \end{enumerate}
\end{definition}

The following lemma from \cite[Lemma 3.2]{SK} introduces the normally ordered product on
$T(R)$ and extends the bracket to the whole tensor algebra $T(R)$.

\begin{lemma} Let $R$ be a non-linear conformal algebra. There exist unique linear maps
\[
N : T(R) \otimes T(R) \longrightarrow T(R), \qquad
L_{\lambda} : T(R) \otimes T(R) \longrightarrow \mathbb{C}[\lambda] \otimes T(R),
\]
such that, for $a,b,c,\dots \in R$ and $A,B,C,\dots \in T(R)$, the following conditions hold:
\begin{align*}
&N(1, A) = N(A, 1) = A, \qquad N(a, B) = a \otimes B, \\
&N(a \otimes B, C) = N(a, N(B, C))
+ N\!\left(\left(\int_{0}^{T} d\lambda\, a\right), L_{\lambda}(B, C)\right) \notag\\
&\quad\qquad\qquad  +  N\!\left(\left(\int_{0}^{T} d\lambda\, B\right), L_{\lambda}(a, C)\right), \\
&L_{\lambda}(1, A) = L_{\lambda}(A, 1) = 0, \qquad L_{\lambda}(a, b) = [a_{\lambda} b], \\
&L_{\lambda}(a, b \otimes C) = N(L_{\lambda}(a,b), C)
+  N(b, L_{\lambda}(a,C)) \\
&\quad\qquad\qquad + \int_{0}^{\lambda} d\mu \, L_{\mu}(L_{\lambda}(a,b), C), \\
&L_{\lambda}(a \otimes B, C) = N\!\left(e^{T\partial_{\lambda}} a, L_{\lambda}(B, C)\right)
+ \, N\!\left(e^{T\partial_{\lambda}} B, L_{\lambda}(a, C)\right) \notag\\
&\quad\qquad\qquad +  \int_{0}^{\lambda} d\mu \, L_{\mu}(B, L_{\lambda-\mu}(a, C)). 
\end{align*}
\end{lemma}
   
Now, we define the subspace  $\M(R) \subseteq \T(R)$ spanned by elements of the form
\[B \otimes \left( \bigl( a \otimes b -  b \otimes a \bigr) \otimes C -{N} \biggl[ \int_{-T}^0 d\lambda \, \mathcal{L}_\lambda(a, b), \, C \biggr]\right),\]
for all $a, b \in R$ and $B, C \in \T(R)$. We denote by $U(R)=\T(R)/\M(R)$

\begin{definition}
    A non-linear Lie {conformal algebra} is a non-linear {conformal algebra} $R$ such that the bracket $[\cdot_\lambda \cdot] : R \otimes R \rightarrow \T(R) [\lambda] $ satisfies
    \begin{enumerate}
        \item[] {Skew-symmetry} $[a_\lambda b] = -[b_{-\lambda-\partial} a]$ , 
        \item[] {Jacobi identity}\,\,\, $[a_\lambda [b_\mu c]] = [b_\mu [a_\lambda c]] + [[a_\lambda b]_{\lambda+\mu} c]$ on  $\CC[\lambda,\mu]\otimes U(R)$. 
   \end{enumerate}
\end{definition}

\section{Celestial Algebras}

Let $\mathfrak{g}$ be a Lie algebra.

\begin{enumerate}
\item
The current Lie algebra $\mathfrak{g}[u,v]=\mathfrak{g}\otimes\mathbb{C}[u,v]$ has bracket
\[
[a\otimes f,\; b\otimes g]=[a,b]\otimes fg.
\]

\item
The adjoint semidirect product Lie algebra $\mathfrak{g}\ltimes\mathfrak{g}_{\mathrm{ad}}$ is given by
\[
[(a,x),(b,y)]=([a,b],[a,y]-[b,x]).
\]
\end{enumerate}

In this work, we consider the following  affine conformal algebra.
\begin{example} \label{exa 2.4}
    Let $\lieg$ be a Lie algebra. The vector space
        $$R_{J,I}: = \C [\partial] \otimes (\mathfrak{g}\ltimes\mathfrak{g}_{\mathrm{ad}})[u,v] \, , $$
    with bracket \eqref{eq1.1} forms a Lie conformal algebra. 
    
    We will consider the notation 
    \[J_a[m,n]:=(a,0)\otimes u^mv^n, \qquad  I_a[m,n]:=(0,a)\otimes u^mv^n\, .\]
    Hence, the bracket  \eqref{eq1.1} can be described as follows:
        \begin{equation}\label{eq2.1}
        \begin{split}
           & [J_a[m,n]_\lambda J_b[t,u]] = J_{[a,b]}[m+t,n+u] , \\
           & [J_a[m,n]_\lambda I_b[t,u]] = I_{[a,b]}[m+t, n+u] , \\
           & [I_a[m,n]_\lambda I_b[t,u]] = 0 .
            \end{split}
        \end{equation}
\end{example}

The algebras introduced in Example \ref{exa 2.4} appear in the context of tree-level gluon scattering amplitudes within pure Yang-Mills theories, see \cite{N,CP}. The following deformation was introduced by \cite{CP}, motivated by an anomaly on self-dual Yang-Mills theories.

Example \ref{exa 2.4} admits the following extension.

\begin{proposition}\label{pro3.2}
Let $\mathfrak{h}=\Span{\{E,F\}}$ be a two-dimensional abelian Lie algebra. The vector space
  $$R_{J,I,E,F} :=R_{J,I}\bigoplus  \C [\partial] \otimes \left({\lieh[u,v]}/{\C E[0,0]}\right)  \, ,  $$
  with brackets defined by \eqref{eq2.1} and 
   \begin{equation}\label{eq2.2}
\begin{split}
    [J_a[m,n]_\lambda E[t,u]] &= \beta  \frac{um-tn}{t+u} I_a[m+t-1, n+u-1] \, , \\
    [J_a[m,n]_\lambda F[t,u]] &= -\beta \left( \lambda  + \frac{m+n}{t+u+2} (\lambda+\partial) \right) I_a[m+t, n+u]\, , 
\end{split}
\end{equation}
where $\beta\in \CC$,  $E[m,n]:=E\otimes u^mv^n$ and $ F[m,n]:=F\otimes u^mv^n$, forms  a Lie conformal algebra.
\end{proposition}

It follows from the previous proposition and Theorem \ref{thm  BK} that 
\begin{equation}\label{eq v}
\V^{\beta}(\lieg):=U(R_{J,I,E,F})
\end{equation}
forms  a vertex  algebra.

\begin{proof} We impose the skew-symmetry of \eqref{eq2.2} by definition. Then by sesquilinearity, it follows that $R_{J,I,E,F}$ is a Lie conformal algebra if the Jacobi identity is satisfied for any triple  
\[(a,b,c)\in X\times X\times X\] 
where $X=\{J_a[n,m],I_a[n,m],E[n,m],F[n,m]\, |\, a\in\lieg\, ,\,\,\, n,m\in\ZZ_{\geq 0}\}$. 

 We proceed as follows. For a triple $(a,b,c)$ we compute (1) $[a_{\lambda}[b_{\mu}c]]$, (2) $[b_{\mu}[a_{\lambda}c]]$ and (3) $[[a_{\lambda}b]_{\lambda+\mu}c]$. Then we verify that $(1)=(2)+(3)$.  
 
  \medskip

Case 1: The Jacobi identity for triples involving only $J$'s and $I$'s follows from Example \ref{exa 2.4}. Additionally, the Jacobi identity for triples containing pairs $(E,F)$, $(I,E)$ and $(I,F)$ are trivially satisfied.

 \medskip

Case 2: The Jacobi identity for  \((J, J, E)\)

The term (1) is 
    \begin{align*}
        [J_a[m,n]&_\lambda [J_b[r,s]_\mu E[t,u]]]  = [J_a[m,n]_\lambda \beta\frac{ru-st}{t+u} I_b[r+t-1, s+u-1]] \\
        &= \beta \frac{ru-st}{t+u}I_{[a,b]} [m+r+t-1, n+s+u-1] \, . 
    \end{align*}
Analogously,  the term $(2)$ is
      \begin{align*}
        [J_b[r,s]_\mu &[J_a[m,n]_\lambda E[t,u]]]   =-\beta \frac{um-tn}{t+u} I_{[a,b]}[m+r+t-1, n+s+u-1]\,.
    \end{align*}
 And $(3)$ is 
    \begin{align*}
        [[J_a[m,n]&_\lambda J_b[r,s]]_{\lambda+\mu} E[t,u]] = [J_{[a,b]} [m+r, n+s]_{\lambda+\mu}E[t,u]] \\
        & = \beta \frac{u(m+r)-t(n+s)}{t+u}I_{[a,b]} [m+r+t-1, n+s+u-1].
    \end{align*}
It is easy to see that $(1)=(2)+(3)$ is satisfied.
    
    \medskip

    {Case 3: The Jacobi identity for  \(J, J, F\).}
    
    The term (1) is 
    \begin{align*}
        [J_a[m,n]&_\lambda [J_b[r,s]_\mu F[t,u]]]  = [J_a[m,n]_\lambda -\beta\left( \mu  + \frac{r+s}{t+u+2} (\mu+\partial) \right) I_b[r+t, s+u]] \\
        &=  -\beta\left( \mu  + \frac{r+s}{t+u+2} (\mu+\lambda+\partial) \right)  I_{[a,b]}[m+r+t,n+ s+u]\, . 
    \end{align*}
    Analogously,  the term $(2)$ is
      \begin{align*}
        [J_b[r,s]_\mu &[J_a[m,n]_\lambda F[t,u]]]   = \beta\left( \lambda  + \frac{m+n}{t+u+2} (\lambda+\mu+\partial) \right)  I_{[a,b]}[m+r+t,n+ s+u]\,.
    \end{align*}
     And $(3)$ is 
    \begin{align*}
        [[J_a[m,n]&_\lambda J_b[r,s]]_{\lambda+\mu} F[t,u]] = [J_{[a,b]} [m+r, n+s]_{\lambda+\mu}F[t,u]] \\
        & =- \beta\left( \lambda+\mu  + \frac{m+r+n+s}{t+u+2} (\lambda+\mu+\partial) \right)I_{[a,b]} [m+r+t, n+s+u].
    \end{align*}
    It is easy to see that $(1)=(2)+(3)$ is satisfied.
\end{proof}

\section{Non-linear deformation}

In this section, we restrict our attention to the generators in \eqref{eq: gen2}. Let $e_i$ be basis of $\lieg$ where $1\leq i\leq \dim \lieg$ and denote its dual basis by $e^i$, that is, $(e_i, e^j)=\delta_{ij}$. 
Here and throughout, repeated indices are understood to be summed over their full range. We consider the following relations:

\begin{equation}\label{eq v2}
\begin{split}
       & [J_a[1,0]_\lambda  J_b[0,1]] = J_{[a,b]}[1,1]-\beta(a,b)((2\lambda+\partial)E[1,1]+F[0,0]) +\\
       &\hspace{0.3cm} D(2\lambda+\partial) I_{[a,b]} [0,0]   +C   \left( J_{[a,e_i]}[0,0]I_{[b,e^i]}[0,0]+ J_{[b,e_i]}[0,0]I_{[a,e^i]}[0,0]\right), \\
        &[J_a[n,m]_\lambda J_b[0,0]] = -\beta(a,b)(n+m)(\lambda+\partial)E[n,m]+J_{[a,b]}[n,m] \,,  \\
        &[J_a[1,0]_\lambda I_b[0,1]] = I_{[a,b]}[1,1] - C I_{[a,e_i]}[0,0]I_{[b,e^i]}[0,0] \,,   \\
       & [J_a[0,1]_\lambda I_b[1,0]] = I_{[a,b]}[1,1] + C I_{[a,e_i]}[0,0]I_{[b,e^i]}[0,0]\, .        
    \end{split}
\end{equation}
where $C,D\in \mathbb{C}$, and in the second equation $n$, $m$ satisfy $n+m\leq 2$.


\begin{theorem}\label{thm4.1}
  \[[J_a[1,0]_\lambda [J_b[0,1]_\mu J_c[0,0]]] =[J_b[0,1]_\mu [J_a[1,0]_\lambda J_c[0,0]]]+[[J_a[1,0]_\lambda J_b[0,1]]_{\lambda+\mu}J_c[0,0]]\]
     if and only if $\beta=D=C=0$ or 
     \begin{align*}
     &\Tr((\ad_a)^4)=\alpha (\Tr((\ad_a)^2))^2\\
&D=\frac{-\beta^2}{8h^{\vee}\alpha}\, ,  \qquad C=\frac{3\beta^2}{8(h^{\vee})^2\alpha}
\end{align*}

\end{theorem}

Before we prove the Theorem, we state the following Lemma

\begin{lemma}\label{lem4.2} Let $\lieg$ be a finite-dimensional simple Lie algebra endowed with a bi-invariant pairing $(.,.)$. Let $\{e_i\}$ be a basis of $\lieg$ with dual basis $\{e^i\}$, i.e. $(e_i,e^j)=\delta_i^{\,j}$. The following conditions hold:
\begin{enumerate}
\item \label{lem4.2 i)} For $a,b,c\in\lieg$,
\[
[[c,[b,e_i]],[a,e^i]]
\;=\;
-\Tr(\ad_a\ad_b\ad_c\ad_{e_i})\,e^i .
\]
\item \label{lem4.2 ii)} Let $D_4=\langle(1 2 3 4), (1 2)(3 4)\rangle \subset S_4$ be the Dihedral group.  For any $\sigma\in D_4$ and for all $a_1,a_2,a_3,a_4\in\lieg$
\[\Tr(\ad_{a_1}\ad_{a_2}\ad_{a_3}\ad_{a_4}) = \Tr(\ad_{a_{\sigma(1)}}\ad_{a_{\sigma(2)}}\ad_{a_{\sigma(3)}}\ad_{a_{\sigma(4)}})\] 
 \item \label{lem4.2 iii)} For all $a,b,c,d\in\lieg$
 \begin{align*}
&2{\Tr}({[\ad_a,\ad_d]} {[\ad_b,\ad_c]}) + {\Tr}({[\ad_a,\ad_b]} {[\ad_c,\ad_d]})\\
&\quad  =4 \big( {\Tr}(\ad_a \ad_b \ad_c \ad_d) +{\Tr}(\ad_a \ad_c \ad_d \ad_b) + {\Tr}(\ad_a \ad_d \ad_b \ad_c) \big)\\
&\quad -6 \big( {\Tr}(\ad_a \ad_b \ad_c \ad_d) + {\Tr}(\ad_b \ad_a \ad_c \ad_d) \big) 
 \end{align*}
\item \label{lem4.2 iv)}There exists $\alpha\in\C$ such that
\[
\Tr\big((\ad_a)^4\big)=\alpha\,\Tr\big((\ad_a)^2\big)^2
\]
 for all $a\in\lieg$ if and only if, for all $a,b,c,d\in\lieg$,
\begin{align*}
&\Tr(\ad_a\ad_b\ad_c\ad_d)+\Tr(\ad_a\ad_c\ad_d\ad_b)+\Tr(\ad_a\ad_d\ad_b\ad_c)\\
&=
\alpha\Big(
\Tr(\ad_a\ad_b)\Tr(\ad_c\ad_d)
+\Tr(\ad_a\ad_c)\Tr(\ad_d\ad_b)
+\Tr(\ad_a\ad_d)\Tr(\ad_b\ad_c)
\Big).
\end{align*}

\end{enumerate}
\end{lemma}

\begin{proof}
(1) If $k^{ij}$ denotes the inverse of $k_{ij}=(e_i,e_j)$, then $e^i=k^{ij}e_j$ and $f_{ae_i}^{e_j}k^{il}=-f_{ae_r}^{e_l}k^{rj}$. Hence, 
\begin{align*}
[[c,[a,e_i]],[b,e^i]]&=f_{ae_i}^{e_k}f_{ce_k}^{e_l}f_{be^i}^{e_r}f_{e_le_r}^{e_m}e_{m}\\
&=-f^{e_i}_{be_n}f_{ae_i}^{e_k}f_{ce_k}^{e_l}f_{e_se_l}^{e_n}e^{s}\\
&=-\Tr(\ad_b\ad_a\ad_c\ad_{e_s})e^s\, .
\end{align*}

(2)-We have that $k^{ir}k_{rj}=\delta_{ij}$, hence  
  \begin{align*}
            \Tr(\ad_{a_1}&\ad_{a_2}\ad_{a_3}\ad_{a_4}) = f_{{a_1}e_j}^{e_i} f_{{a_2}e_i}^{e_k} f_{{a_3}e_k}^{e_l} f_{{a_4}e_l}^{e_j}  \\
                &= (f_{{a_1}e_m}^{e_i}k^{mn}k_{nj}) (f_{{a_4}e_p}^{e_j}k^{pq}k_{ql}) (f_{{a_3}e_r}^{e_l}k^{rs}k_{sk}) (f_{{a_2}e_t}^{e_k}k^{tu}k_{ui})   \\
                &= (-f_{{a_1}e_o}^{e_n}k^{oi}k_{nj}) (-f_{{a_4}e_v}^{e_q}k^{vj}k_{ql}) (-f_{{a_3}e_w}^{e_s}k^{wl}k_{sk}) (-f_{{a_2}e_x}^{e_u}k^{xk}k_{ui}) \\
                &= f_{{a_1}e_o}^{e_n} f_{{a_4}e_n}^{e_q} f_{{a_3}e_q}^{e_s} f_{{a_2}e_s}^{e^o}= \Tr(\ad_{a_1}\ad_{a_4}\ad_{a_3}\ad_{a_2})= \Tr(\ad_{a_2}\ad_{a_1}\ad_{a_4}\ad_{a_3}).
        \end{align*}

(3)  The identity was presented in \cite[Appendix B]{CP2}  with a minor typo. We reproduce it here for completeness. First, using cyclicity of traces 
\begin{align*}
&3(\Tr(\ad_a\ad_b\ad_c\ad_d)+\Tr(\ad_b\ad_a\ad_c\ad_d))\\
&-2(\Tr(\ad_a\ad_b\ad_c\ad_d)+\Tr(\ad_a\ad_c\ad_d\ad_b)+\Tr(\ad_a\ad_d\ad_b\ad_c)) \\
&=\Tr([\ad_d,\ad_a]\ad_b\ad_c)+\Tr([\ad_a,\ad_c]\ad_d\ad_b)
\end{align*}

Next, notice that by Lemma \ref{lem4.2} (\ref{lem4.2 ii)}) we have that $\Tr([\ad_d,\ad_a]\ad_b\ad_c)=\frac{1}{2}\Tr([\ad_d,\ad_a][\ad_b,\ad_c])$. And
\begin{align*}
  &  \Tr([\ad_d,\ad_a]\ad_b\ad_c)+\Tr([\ad_a,\ad_c]\ad_d\ad_b) \\
&=\frac{1}{2}\Tr([\ad_d,\ad_a][\ad_b,\ad_c])+\frac{1}{2}\Tr([\ad_a,\ad_c][\ad_d,\ad_b]) \\
&=-\Tr([\ad_a,\ad_d][\ad_b,\ad_c])-\frac{1}{2}\Tr([\ad_a,\ad_b][\ad_c,\ad_d])\, .
\end{align*}
Where in the last equality, we used trace invariance and Jacobi identity.

(4) If $\Tr\big((\ad_a)^4\big)=\alpha\,\Tr\big((\ad_a)^2\big)^2$, we have that
\begin{align}
&\Tr\big((\ad_{a_1+a_2})^4\big)=\alpha\,\Tr\big((\ad_{a_1+a_2})^2\big)^2 \, , \label{eq4.1 4}\\
&\Tr\big((\ad_{a_1+a_2+a_3})^4\big)=\alpha\,\Tr\big((\ad_{a_1+a_2+a_3})^2\big)^2 \, , \label{eq4.2 4}\\
&\Tr\big((\ad_{a+b+c+d})^4\big)=\alpha\,\Tr\big((\ad_{a+b+c+d})^2\big)^2  \, . \label{eq4.3 4}
\end{align}
Then expanding \eqref{eq4.1 4} and \eqref{eq4.2 4} (using the linearity of $\ad:\lieg\rightarrow \lieg$) and substituting into \eqref{eq4.3 4}, we obtain the identity. The other direction follows from replacing $a=b=c=d$.
\end{proof}

Finally, we prove the main theorem of this work

\begin{proof}[Proof of Theorem \ref{thm4.1} ]

It follows directly by definition that the brackets are skew-symmetric. Then by sesquilinearity, we have that $R$ is a non-linear Lie conformal algebra if the Jacobi identity is satisfied for any triple  
\[(a,b,c)\in U\times U\times U\, .\]

 We proceed as follows.\\ 
 Step 1: For a triple $(a,b,c)$ we compute (1) $[a_{\lambda}[b_{\mu}c]]$, (2) $[b_{\mu}[a_{\lambda}c]]$ and (3) $[[a_{\lambda}b]_{\lambda+\mu}c]$. \\
 Step 2: Then we verify that $(1)=(2)+(3)$.\\

 The Jacobi identity for  \((J_a[1,0], J_b[0,1], J_c[0,0])\)

The term (1) gives us 
    \begin{align*}
        [J_a&[1,0]_\lambda [J_b[0,1]_\mu J_c[0,0]]] = [J_a[1,0]_\lambda -\beta(b,c)(\mu+\partial) E[0,1])+J_{[b,c]}[0,1]]\\
        &=J_{[a,[b,c]]}[1,1]-\beta^2(b,c)(\mu+\lambda+\partial)I_a[0,0]-\beta(a,[b,c])((2\lambda+\partial)E[1,1]+F[0,0])\\
        & +D(2\lambda+\partial )I_{[a,[b,c]]}[0,0]+C  \left( J_{[a,e_i]}[0,0]I_{[[b,c],e^i]}[0,0] + J_{[[b,c],e_i]}[0,0]I_{[a,e^i]}[0,0]\right).
    \end{align*}
    
Analogously, the term (2) is
    \begin{align*}
        [J_b&[0,1]_\mu [J_a[1,0]_\lambda J_c[0,0]]] =[J_b[0,1]_\mu -\beta(a,c)(\lambda+\partial) E[1,0]+J_{[a,c]}[1,0]]\\
        &=J_{[b,[a,c]]}[1,1]+\beta^2(a,c)(\mu+\lambda+\partial)I_b[0,0]-\beta(b,[a,c])((2\mu+\partial)E[1,1]-F[0,0])\\
        &-D(2\mu+\partial ) I_{[b,[a,c]]}[0,0] -C \left( J_{[[a,c],e_i]}[0,0]I_{[b,e^i]}[0,0] + J_{[b,e_i]}[0,0]I_{[[a,c],e^i]}[0,0]\right).
    \end{align*}
    
And $(3)$ is given by 
\begin{align*}
    &[[J_a[1,0]_\lambda J_b[0,1]]_{\lambda+\mu}J_c[0,0]] =[J_{[a,b]}[1,1]-\beta(a,b)((2\lambda+\partial) E[1,1]+F[0,0])\\
    & \,+D(2\lambda +\partial)I_{[a,b]}[0,0]+C(J_{[a,e_i]}[0,0]I_{[b,e^i]}[0,0]+J_{[b,e_i]}[0,0]I_{[a,e^i]}[0,0])_{\lambda+\mu}J_c[0,0]]\\
    &\hspace{.4cm}=J_{[[a,b],c]}[1,1]+\beta^2(a,b)(\lambda+\mu+\partial)I_{c}[0,0]-\beta ([a,b],c)2(\lambda+\mu+\partial)E[1,1]\\
    &\hspace{.4cm} +D\left({\mu}-{\lambda}\right) I_{[c,[a,b]]}[0,0]+C(\lambda+\mu+\partial)(I_{[[c,[a,e_i]],[b,e^i]]}+I_{[[c,[b,e_i]],[a,e^i]]})\\
    &\hspace{.4cm} -C \left( J_{[c,[a,e_i]]}[0,0]I_{[b,e^i]}[0,0] + J_{[a,e_i]}[0,0]I_{[c,[b,e^i]]}[0,0]\right)\\
    &\hspace{.4cm} -C \left( J_{[c,[b,e_i]]}[0,0]I_{[a,e^i]}[0,0] + J_{[b,e_i]}[0,0]I_{[c,[a,e^i]]}[0,0]\right).
\end{align*}
where we used the identity \eqref{eq2.0}. Now we will group corresponding terms and simplify the expressions. 

The equality $(1)=(2)+(3)$ for the terms with coefficient $J[1,1]$ cancel from the Jacobi identity.

The equality $(1)=(2)+(3)$ for the terms with coefficient $E$, $F$ follows from 
\begin{align*}
& -\beta (a,[b,c])((2\lambda+\partial)E[1,1]+F[0,0])=\\
 &\quad -\beta (b,[a,c])((2\mu+\partial)E[1,1]-F[0,0]) -\beta ([a,b],c)2(\lambda+\mu+\partial)E[1,1]. 
\end{align*}

The equality $(1)=(2)+(3)$ for the terms with coefficient $JI$ requires proving the following identity 
\begin{align*}
    &C \left( J_{[a,e_i]}[0,0]I_{[[b,c],e^i]}[0,0] + J_{[[b,c],e_i]}[0,0]I_{[a,e^i]}[0,0]\right) \\
    &=-C  \left( J_{[[a,c],e_i]}[0,0]I_{[b,e^i]}[0,0] + J_{[b,e_i]}[0,0]I_{[[a,c],e^i]}[0,0]\right) \\
    &-C \left( J_{[c,[b,e_i]]}[0,0]I_{[a,e^i]}[0,0] + J_{[b,e_i]}[0,0]I_{[c,[a,e^i]]}[0,0]\right) \\
    & -C \left( J_{[c,[a,e_i]]}[0,0]I_{[b,e^i]}[0,0] + J_{[a,e_i]}[0,0]I_{[c,[b,e^i]]}[0,0]\right)\, . 
\end{align*}
From the Jacobi identity the equality above is equivalent to
\begin{align*}
    &J_{[a,e_i]}[0,0]I_{[b,[e^i,c]]}[0,0] + J_{[b,[e_i,c]]}[0,0]I_{[a,e^i]}[0,0] \\
    &+J_{[a,[e_i,c]]}[0,0]I_{[b,e^i]}[0,0] + J_{[b,e_i]}[0,0]I_{[a,[e^i,c]]}[0,0]=0\, .
\end{align*}
Hence, for all $x,y\in \lieg$ the coefficients of $J_x[0,0]I_{y}[0,0]$ satisfy   
\begin{align*}
    (x,[a,e_i])&(y,[b,[e^i,c]])+ (x,[b,[e_i,c]])(y,[a,e^i]) \\
    +&(x,[a,[e_i,c]])(y,[b,e^i]) + (x,[b,e_i])(y,[a,[e^i,c]])=0\, .
\end{align*}
And using $(e_i, *)(e^i, *)=(*,*)$ we obtain
\begin{align*}
    &([x,a],[c,[y,b]])+ ([c,[x,b]],[y,a]) +([c,[x,a]],[y,b]) + ([x,b],[c,[y,a]])=0\, ,
\end{align*}
which is satisfied by bi-invariance.  \\
 
Finally, the equality $(1)=(2)+(3)$ for the terms with coefficient $I$ requires proving the following identity 
\begin{align*}
    -\beta^2(b,c)&(\mu+\lambda+\partial)I_a[0,0] +D(2\lambda+\partial)I_{[a,[b,c]]}[0,0] = \beta^2(a,c)(\mu+\lambda+\partial)I_b[0,0] \\
    &-D(2\mu+\partial) I_{[b,[a,c]]}[0,0] +\beta^2(a,b)(\lambda+\mu+\partial)I_{c}[0,0] \\
    &+D({\mu}-{\lambda}) I_{[c,[a,b]]}[0,0] +C (\lambda+\mu+\partial)(I_{[[c,[a,e_i]],[b,e^i]]}+I_{[[c,[b,e_i]],[a,e^i]]})\, .
\end{align*}
Applying the Jacobi identity to $I_{[b,[a,c]]}[0,0]$ and grouping terms we obtain 
\begin{align*}
    &2D(\lambda+\mu+\partial)I_{[a,[b,c]]}[0,0] +D({\lambda}+{\mu}+\partial) I_{[c,[a,b]]}[0,0] \\
    = &\beta^2(\mu+\lambda+\partial)((b,c)I_a[0,0] + (a,c)I_b[0,0] +(a,b)I_{c}[0,0]) \\
       &+C (\lambda+\mu+\partial)(I_{[[c,[a,e_i]],[b,e^i]]}+I_{[[c,[b,e_i]],[a,e^i]]}) \, .
\end{align*}
Hence, for all $x\in \lieg$ the coefficients of $I_x[0,0]$ satisfy   
\begin{align*}
    2D(x,[a,[b,c]]) &+D (x,[c,[a,b]]) =\beta^2((b,c)(x,a) + (a,c)(x,b) +(a,b)({c},x)) \\
    &+C ((x,[[c,[a,e_i]],[b,e^i]])+{(x,[[c,[b,e_i]],[a,e^i]])})\, .
\end{align*}
Applying \eqref{fq2.1} and Lemma \ref{lem4.2} (\ref{lem4.2 i)}) we obtain
\begin{align*}
    &\frac{D}{h^{\vee}}\Tr([\ad_x,\ad_a][\ad_b,\ad_c]) +\frac{D}{2h^{\vee}} \Tr([\ad_x,\ad_c][\ad_a,\ad_b]) \\
    &=\beta^2((b,c)(x,a) + (a,c)(x,b) +(a,b)({c},x)) -C (\Tr(\ad_b\ad_a\ad_c\ad_{x})+\Tr(\ad_a\ad_b\ad_c\ad_{x}))
\end{align*}
Applying Lemmas \ref{lem4.2} (\ref{lem4.2 ii)}) and (\ref{lem4.2 iii)}) we obtain 
 \begin{align*}
&(C+\frac{3D}{h^\vee}) \big( \operatorname{Tr}(\ad_a \ad_b \ad_c \ad_x) + \operatorname{Tr}(\ad_b \ad_a \ad_c \ad_x) \big)\\
&=\beta^2((b,c)(x,a) + (a,c)(x,b) +(a,b)({c},x))\\
&+ \frac{2D}{h^\vee} \big( \operatorname{Tr}(\ad_a \ad_b \ad_c \ad_x) + \operatorname{Tr}(\ad_a \ad_c \ad_x \ad_b) + \operatorname{Tr}(\ad_a \ad_x \ad_b \ad_c) \big)\, .
 \end{align*}

 Now, the elements $e,(243),(423)$ are representatives of the equivalence classes of $S_4/D_4$. Therefore, the coefficient of $\frac{2D}{h^\vee}$ is $S_4$ symmetric using Lemma \ref{lem4.2} (\ref{lem4.2 ii)}). Also, the coefficient of $\beta^2$ is $S_4$ symmetric. On the other hand, the coefficient of $(C+\frac{3D}{h^\vee})$ is not $S_4$ symmetric, in particular it is not preserved by the permutation $(23)$\footnote{By contradiction, consider the permutation $(23)$ and $a=c$, $b=d$. Then the identity 
 \[\operatorname{Tr}(\ad_a \ad_b \ad_a \ad_b) + \operatorname{Tr}(\ad_b \ad_a \ad_a \ad_b)=\operatorname{Tr}(\ad_a \ad_a \ad_b \ad_b) + \operatorname{Tr}(\ad_b \ad_a \ad_a \ad_b)\]
implies that $\operatorname{Tr}(\ad_{[a,b]} \ad_a \ad_b) = 0$ for all $a,b\in \lieg$. This is a contradiction, for example, by taking \(a = e\) and \(b = f\) in an \(\mathfrak{sl}_2\)-triple inside \(\mathfrak{g}\).}. Hence, subtracting  the expression above with the same expression exchanging $b$ and $c$ we obtain

 \begin{align*}
&(C+\frac{3D}{h^\vee})=0\\
&\beta^2((b,c)(x,a) + (a,c)(x,b) +(a,b)({c},x))\\
&=\frac{-2D}{h^\vee} \big( \operatorname{Tr}(\ad_a \ad_b \ad_c \ad_x) + \operatorname{Tr}(\ad_a \ad_c \ad_x \ad_b) + \operatorname{Tr}(\ad_a \ad_x \ad_b \ad_c) \big)
 \end{align*}
Finally from Lemma \ref{lem4.2} (\ref{lem4.2 iv)}) we have  
\[D=\frac{-\beta^2}{8h^{\vee}\alpha}\, ,  \qquad C=\frac{3\beta^2}{8(h^{\vee})^2\alpha}\, .\]

\end{proof}

Finally, using the previous Theorem \ref{thm4.1} and the following Lemma we obtain the Theorem \ref{theorem: 1.1}.
\begin{lemma}
    There exists an $\alpha\in \CC$ such that 
    \begin{equation}\label{eq:trace-identity}
        \Tr((\ad_a)^4) = \alpha(\Tr((\ad_a)^2))^2
    \end{equation}
    if and only if $\lieg \in \{A_1, A_2, D_4,  E_6, E_7, E_8, F_4, G_2\}$. Further, 
    \begin{equation}\label{eq:alpha-expression}
        \alpha=\frac{5}{2(2+\dim \lieg)}.
    \end{equation}
\end{lemma}

\begin{proof} 
For types $A_1, A_2, E_6, E_7, E_8, F_4$ and $G_2$, \eqref{eq:trace-identity} follows from the degrees of the fundamental invariants. And it was proved in \cite{O} that the identity \eqref{eq:alpha-expression} is satisfied for exceptional Lie algebras. Now, for classical types we have that 
\begin{table}[h!]
    \centering
    \begin{tabular}{c|c}
         $A_n$ & $\Tr(\ad_a^4) = 2(n+1)\Tr(a^4) + 6(\Tr(a^2))^2$  \\
         $B_n$ & $\Tr(\ad_a^4) = (2n-7)\Tr(a^4) + 3(\Tr(a^2))^2 $ \\
         $C_n$ & $\Tr(\ad_a^4) = 2(n+4)\Tr(a^4) + 3(\Tr(a^2))^2$ \\
         $D_n$ & $\Tr(\ad_a^4) = 2(n-4)\Tr(a^4) + 3(\Tr(a^2))^2$.
    \end{tabular}
\end{table}

Moreover, $\operatorname{Tr}(a^4)$ is an invariant generator, except for types $A_1$ and $A_2$. And, on $D_4$ the coefficient of $\Tr(a^4)$ is canceled. Hence, $A_1, A_2, D_4$ and the exceptional cases are all possible cases where \eqref{eq:trace-identity} is satisfied. Finally, it follows from a direct calculation that \eqref{eq:alpha-expression}  is also satisfied for $A_1, A_2, D_4$. 
\end{proof}

\bibliographystyle{amsalpha}

\end{document}